\documentclass{IEEEtran}
\usepackage{cite}
\usepackage{amsmath,amssymb,amsfonts}
\usepackage{graphicx}
\usepackage{textcomp}
\usepackage{epstopdf,color}
\usepackage{algorithm}  
\usepackage{algorithmic}  
\newtheorem{theorem}{Theorem}
\newtheorem{remark}{Remark}
\newtheorem{corollary}{Corollary}
\newtheorem{definition}{Definition}

\newtheorem{assumption}{Assumption}
\newtheorem{problem}{Problem}
\def\no{\nonumber}
\def\BibTeX{{\rm B\kern-.05em{\sc i\kern-.025em b}\kern-.08em
    T\kern-.1667em\lower.7ex\hbox{E}\kern-.125emX}}
\begin{document}
\title{Spectrum Assignment of Stochastic Systems with Multiplicative Noise}
\author{Xiaomin~Xue,~
        Juanjuan~Xu,~
        and~Huanshui~Zhang
\thanks{Xiaomin Xue and Juanjuan Xu are with the School of Control Science and Engineering, Shandong University, Jinan, 250061, China (e-mail: xmxue96@163.com, juanjuanxu@sdu.edu.cn).}
\thanks{Huanshui Zhang is with College of Electrical Engineering and Automation, Shandong University of Science and Technology, Qingdao, 266000, China (e-mail: hszhang@sdu.edu.cn).}
}

\maketitle

\begin{abstract}
This paper studies the spectrum assignment of a class of stochastic systems with multiplicative noise.  A novel $\alpha$-spectrum assignment is proposed for discrete-time and continuous-time stochastic systems with multiplicative noise. In particular, $0$-spectrum assignment is equivalent to the pole assignment for the deterministic systems. The main contribution is two-fold: On the one hand, we present the conditions for $\alpha$-spectrum assignment and the design of feedback controllers based on the system parameters. On the other hand, when the system parameters are unknown, we present a stochastic approximation algorithm to learn the feedback gains which guarantee the spectrum of the stochastic systems to achieve the predetermined value. Numerical examples are provided to demonstrate the effectiveness of the proposed algorithms.
\end{abstract}


\section{Introduction}
\label{sec:introduction}
The location of the poles directly influences the dynamic performance of the system, including stability, damping characteristics, transient response, and so on (see \cite{lam1999pole,kim1990optimal,zhang2021coupling}). Thus pole assignment becomes one of the fundamental problems in modern control theory, which aims to design appropriate feedback gains such that the poles of the closed-loop system are assigned to desired locations, and has been widely applied in various engineering fields such as power systems, robot dynamics, aircraft systems (see \cite{chow1989pole,takagi2005robust,steed2016algebraic}).

For deterministic systems, plenty of significant progress has been made for the pole assignment problem. For instance, \cite{wonham1967pole} studied the pole assignment problem for continuous-time linear systems, and demonstrated that the pole assignability under the state feedback controller is equivalent to completely controllability of system. For an arbitrarily given set of desired poles (real or complex conjugate), \cite{antsaklis1997linear} provided several methods to achieve pole assignment under the assumption of controllability, including the direct method, the controllable canonical form method, and the Ackermann's formula.
\cite{kimura1975pole} studied the pole assignment problem with incomplete state observation by using the output feedback controller, and presented the assignable conditions for the closed-loop poles.

Different from the deterministic case, when the system is affected by stochastic disturbances, particularly multiplicative noise, the classical formulation of pole assignment encounters theoretical challenges, as the notion of poles becomes inadequate.
As a result, the concept of spectrum assignment has been generalized to characterize the dynamic performance of stochastic systems. By analyzing the spectrum of closed-loop operator for stochastic system, \cite{zhang2004stabilizability} obtained that the continuous-time stochastic system is mean-square asymptotically stabilizable if and only if the operator spectrum lies in the open left-half complex plane. In \cite{zhang2009interval}, the interval stability of continuous-time stochastic systems was further studied, and a sufficient condition is provided for assigning the real parts of the operator spectrum within a specified interval. \cite{hou2011study} studied the regional stability of discrete-time stochastic system, and assigned the operator spectrum within given geometric regions in complex plane, such as circular region, sector region, annulus region.

Different from the existing results that assign the spectrum to certain domains with strictly unequal upper and lower bounds, which cannot degenerate into the case of a single point, in this paper, we will study the spectrum assignment to an exact location for a class of stochastic systems with multiplicative noise.
To this end, we propose a innovative definition of $\alpha$-spectrum assignment, where $0$-spectrum assignment is equivalent to the pole assignment when systems are reduced to the deterministic case. The main contribution is consisting of two aspects: On the one hand, we establish the assignable conditions for $\alpha$-spectrum assignment and give the design of feedback controllers based on system parameters. On the other hand, when the model parameters are unknown, we design a stochastic approximation algorithm to learn the feedback gains ensuring the spectrum of the stochastic systems to achieve the predetermined value.

The rest of the paper is organized as follows. Section \ref{sec2} states the spectrum assignment problem to be addressed in this paper. The main results and the corresponding stochastic approximation-based learning algorithm are presented for discrete-time and continuous-time system in Section \ref{sec3} and Section \ref{sec4}, respectively.  Section \ref{sec5} provides numerical examples to demonstrate the effectiveness of the designed algorithms.
Section \ref{sec6} concludes this paper.

The following notations will be used in this paper. $A'$ denotes the transpose of the matrix $A$. $\mathbb{R}$ and $\mathbb{R}^n$ $(n>1)$ denote 1-dimensional and n-dimensional Euclidean space, respectively. $\mathbb{S}^n$ denotes the set of all $n \times n$ symmetric positive definite matrices. $\mathbb{C}$ denotes the set of complex numbers. $I$ denotes the identity matrix. $\text{Rank}(A)$ denotes the rank of the matrix $A$.

\section{Problem Statement}\label{sec2}
Consider the following discrete-time stochastic systems with multiplicative noise:
\begin{align}\label{1}
\left\{\aligned
&x(k+1) =  H x(k) + L u(k) + F v(k) + u(k) w(k), \\
&x(0) =  x_0,\\
\endaligned\right.
\end{align}
where \( x(k) \in \mathbb{R}^n \) is the state, \( u(k) \in \mathbb{R}^n \) and \( v(k) \in \mathbb{R} \) are control inputs. $ H, L, F$ are constant matrices of compatible dimensions. $w(k)$ is a scalar-valued white noise with a mean of 0 and a variance of $\delta$ on a complete probability space $(\Omega, P, \mathcal{F}, \mathcal{F}_{k})$ where $\mathcal{F}_{k}=\sigma\{w_0, \ldots, w_k\}$.

The aim of the paper is to study the pole assignment of the stochastic system (\ref{1}). However, different from the deterministic systems, there exist multiple kinds of stabilizations, e.g., the mean-square stabilization, the almost sure
stabilization, and so on. This leads to that the pole assignment of stochastic system (\ref{1}) is much more complicated than the deterministic case.

\subsection{Preliminaries}\label{sec2.1}

Recalling the deterministic case, i.e., $u(k)=0$ in (\ref{1}), the definition of the pole assignment is to find a feedback gain matrix $K$ such that the eigenvalue of the matrix $H+FK$ can be arbitrarily given, in addition, the state feedback and the output feedback controllers can change the positions of the poles \cite{wonham1967pole,kimura1975pole}.

However, considering the involvement of control-dependent noise, the extension of the corresponding definition to stochastic system (\ref{1}) is non-trivial. To this end, we firstly introduce the following operator:
\begin{align}\label{8}
\mathcal{L}: X \mapsto & (H + L K_u + F K_v) X (H + L K_u + F K_v)' \no\\
& +  K_u X K_u', \quad \forall X \in \mathbb{S}^n,
\end{align}
where $K_u, K_v$ are constant matrices with compatible dimensions.

Then, we define the spectrum of operator $\mathcal{L}$ as the following set:
\begin{align}\label{9}
\sigma(\mathcal{L}) = \left\{ \mu \in \mathbb{C}: \mathcal{L}(X) = \mu X, \  X \in \mathbb{S}^n, \  X \neq 0 \right\}.
\end{align}
In particular, the spectrum $\sigma(\mathcal{L})$ can be obtained by solving the eigenvalues of the following matrix:
\begin{align}\label{m9}
&(M' M)^{-1} M' [(H+L K_u+F K_v) \otimes (H+L K_u+F K_v)\no\\
&+ K_u \otimes K_u ] M,
\end{align}
where $M$ is defined by $vec(X) = M \overline{vec}(X),$ $vec(X)$ denotes the column vector obtained by stacking all the elements of $X \in \mathbb{S}^n$ in a column-major order, $\overline{vec}(X)$ denotes the column vector obtained by stacking the elements of the lower triangular part of $X \in \mathbb{S}^n$ in a column-major order.

\begin{remark}\label{r1}
It is noted that all the spectrum of operator $\mathcal{L}$ lies in the unit circle is equivalent to the fact that the stochastic system (\ref{1}) is mean-square stable \cite{hou2011study}, under the state feedback controller
 \begin{align}\label{6}
u(k) &= K_u x(k), \quad v(k) = K_v x(k).
\end{align}
In this case, the closed-loop system of \eqref{1} becomes:
\begin{align}\label{7}
\begin{cases}
x(k+1) \!=\! (H \!+\! L K_u \!+\! F K_v) x(k) \!+\!  w(k) K_u x(k), \\
x(0) = x_0.
\end{cases}
\end{align}
\end{remark}

\subsection{Definition of spectrum assignment}\label{sec2.2}

Based on Remark \ref{r1}, following the definitions of pole assignment in deterministic case \cite{antsaklis1997linear} and spectrum assignment in stochastic case \cite{zhang2009interval}, we now present the definition of the spectrum assignment for stochastic system (\ref{1}).
\begin{definition}\label{def}
The spectrum assignment for stochastic system (\ref{1}) means that for given $\mu_s, s=1,\ldots, \frac{n(n+1)}{2}$, there exist matrices $K_u$ and $K_v$
such that the spectrum of operator $\mathcal{L}$ is $\mu_s, s=1, \ldots, \frac{n(n+1)}{2}$.
\end{definition}

However, Definition \ref{def} motivated from deterministic case is not established for the stochastic system. The main reason lies in that the spectrum of the operator $\mathcal{L}$ can not be arbitrarily assigned in general. Taking the scalar system $n=1$ for example, the spectrum of operator $\mathcal{L}$ must be nonnegative, that is, one cannot find any $K_u$ and $K_v$ such that the spectrum of operator $\mathcal{L}$ is $\mu_1<0$.

To this end, we introduce a novel $\alpha$-spectrum assignment for stochastic system (\ref{1}) as follows.

\begin{definition}\label{def2}
The $\alpha$-spectrum assignment for stochastic system (\ref{1}) means that for any $\alpha \in \mathbb{R}$, and $\lambda_i, i=1,\ldots, n$, where $\lambda_i$ are either real numbers or conjugate complex numbers, there exist matrices $K_v$ and $K_u=\alpha I$ such that for given spectrum $\mu_s^* , s=1, \ldots, \frac{n(n+1)}{2}$ of operator $\mathcal{L}$ satisfies $\mu_s^* = \lambda_i \lambda_j + \alpha^2,\quad  j\geq i, \  i, j = 1, \dots, n$.
\end{definition}

\begin{remark}
According to Definition \ref{def2}, $0$-spectrum assignment is equivalent to the pole assignment for the deterministic system. In fact, the operator $\mathcal{L}$ in the deterministic case is reduced to
\begin{align}\label{d8}
\mathcal{L}: X \mapsto (H+ F K_v) X (H+ F K_v)', \quad \forall X \in \mathbb{S}^n .
\end{align}
It is easily obtained that the structure of the spectrum for the operator $\mathcal{L}$ must be in the form of $\mu_s=\lambda_i\lambda_j$, where $\lambda_i, i=1, \ldots, n$ are the eigenvalues of the matrix $H+FK_v$. Accordingly, the $\alpha$-spectrum assignment and the pole assignment are equivalent for the deterministic systems.
\end{remark}


Based on the above definitions, we finally state the addressed problem as follows.
 \begin{problem}\label{P1}
Find $K_v$ to achieve $\alpha$-spectrum assignment for stochastic system (\ref{1}).
 \end{problem}

To conclude this section, we make some explanations for the studied system (\ref{1}).
 \begin{remark}
Plenty of practical application scenarios can be modeled by the system \eqref{1}, for example in autonomous vehicles, the system is controlled by two control inputs $u$ and $v$ through a communication network, where $v$ in system \eqref{1} represents the local controller, and $u$ represents the remote controller that may encounters disturbances (\cite{Asghari2018optimal}), as shown in Fig.\ref{pole}.
\end{remark}

\begin{figure}[h]
  \centering
  \includegraphics[width=8cm]{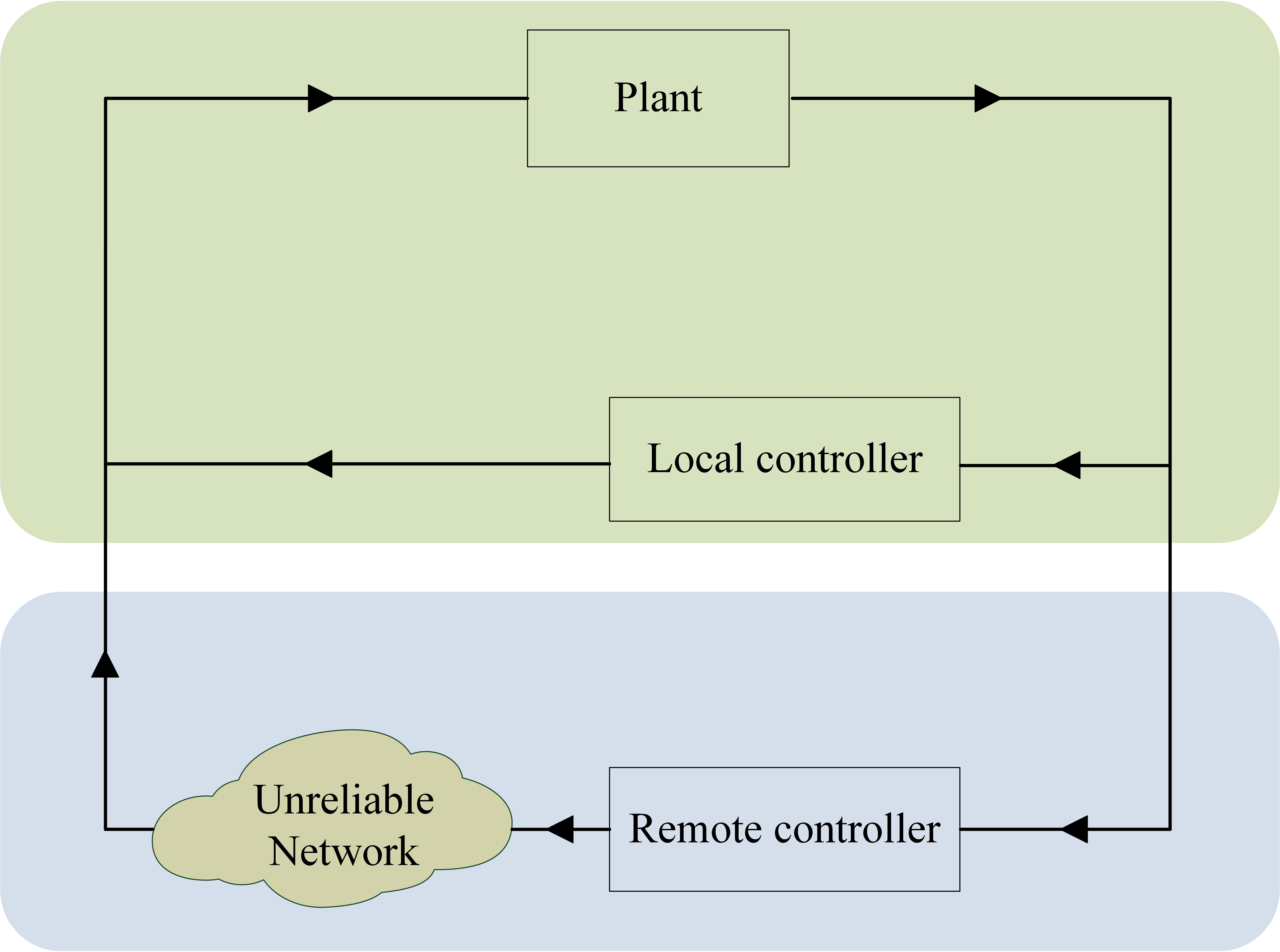}\\
  \caption{Communication network with local controller and remote controller}\label{pole}
\end{figure}

\begin{remark}\label{rem1}
System \eqref{1} can also be derived from the linear stochastic system $( A, B, \bar{A}, \bar{B} )$ with state- and control-dependent noise:
\begin{align}\label{2}
\left\{\aligned
&x(k+1)\!=\! A x(k) \!+\! B U(k) \!+\! (\bar{A} x(k) \!+\! \bar{B} U(k)) w(k), \\
&x(0)= x_0.
\endaligned\right.
\end{align}
 In fact, when the rank of the matrix $\bar{B}$ is equal to $n$, there exists an invertible matrix $Q$ such that $\bar{B} Q = \begin{bmatrix} I & 0 \end{bmatrix}.$ In this case, by defining
 $ BQ \triangleq \begin{bmatrix} L & F \end{bmatrix}, \quad A - L \bar{A} \triangleq  H,$ $
U(k) = \begin{bmatrix} q(k) \\ v(k) \end{bmatrix},$
and
$u(k) = \bar{A} x(k) + q(k),$ system \( (A, B, \bar{A}, \bar{B}) \) can be reformulated as system \eqref{1}.
\end{remark}

\section{Main Result}\label{sec3}
In this section, we give the main result for Problem 1 by dividing into two parts. The first part is to present the condition to guarantee the feasibility of spectrum assignment and the design of the feedback gain matrices. The second part is to design a stochastic approximation algorithm for the feedback gain matrices when the system matrices in (\ref{1}) are unknown.

\subsection{Spectrum Assignment}

To guarantee the feasibility of the spectrum assignment, we make the following controllability assumption:
\begin{assumption}\label{Ass1}
The system ($ G, F$) is completely controllable, where $G = H + \alpha L$ for the given $\alpha$ in Problem \ref{P1} and matrices $L, F$ in system \eqref{1}.
\end{assumption}

The main result of Problem \ref{P1} is then given as follows:
 \begin{theorem}\label{Thm1}
 Under Assumption \ref{Ass1}, the spectrum \eqref{9} generated by system \eqref{1} can be assigned to the desired set \( \{ \mu_1^*, \dots, \mu_\frac{n(n+1)}{2}^*\}  \) according to Definition \ref{def2}. In this case, the feedback gain \( K_v \)  is given by
\begin{align}
K_v =  a D + b G^n,\label{13}
\end{align}
where $a=\begin{bmatrix}a_{n-1}^* & \cdots & a_1^* & a_0^*\end{bmatrix}  $ denotes the coefficient of the following polynomial:
\begin{align}
&\prod_{i=1}^{n} (\beta  - \lambda_i) = \beta^n + a_{n-1}^* \beta^{n-1} + \cdots + a_1^* \beta + a_0^*, \label{16}
\end{align}
and
\begin{align}
b =&- \begin{bmatrix} 0&  \cdots & 0& 1\end{bmatrix}\begin{bmatrix} F& G F&  \dots &G^{n-1} F\end{bmatrix}^{-1}, \label{14}\\
D=& \begin{bmatrix} (b G^{n-1})' &  \cdots & (b G)'& b'\end{bmatrix}'.\label{166}
\end{align}

 \end{theorem}

\begin{proof}
It is sufficient to show that there exists a nonzero matrix $X \in \mathbb{S}^n$ such that $\mathcal{L}(X) = \mu X$ in \eqref{8}-\eqref{9} for $\mu$ belonging to the desired set $\{ \mu_1^*, \dots, \mu_\frac{n(n+1)}{2}^*\}$.
By using the completely controllability of system ($ G, F$)
and Ackermann's formula in \cite{Sontag2013mathematical},
there exists nonzero vectors $\xi_i \in R^n$ such that
\begin{align}\label{A4}
( G + F K_v) \xi_i = \lambda_i \xi_i.
 \end{align}
for $i=1, 2, ..., n$ with $K_v$ defined as \eqref{13}-\eqref{166}.
Define
\begin{align}\label{A6}
X \in \{X_s| X_s=\ &\xi_i \xi_j'+\xi_j \xi_i',\ j\geq i, \  i, j = 1, \dots, n, \no\\
&s=1, \ldots, \frac{n(n+1)}{2}\},
 \end{align}
 it is observed that $X \in \mathbb{S}^n.$
 Together with the definition of $\mu$ in Definition \ref{def2}, it is obtained from \eqref{A4} that
\begin{align}\label{A3}
( G + F K_v)  X ( G + F K_v)' = (\mu - \alpha^2) X,
 \end{align}
 where $\mu \in \{ \mu_1^*, \dots, \mu_\frac{n(n+1)}{2}^*\} .$
By combining with $K_u = \alpha I$ and \eqref{A3}, we obtain that
\begin{align}\label{A2}
 &(H + L K_u + F K_v) X (H + L K_u + F K_v)' +  K_u X K_u' \no\\
 &= \mu X,
\end{align}
for all $\mu$ belong to the desired set $\{ \mu_1^*, \dots, \mu_\frac{n(n+1)}{2}^*\}$. This completes Theorem \ref{Thm1}.
\end{proof}

According to Theorem \ref{Thm1}, we further derive the result of spectrum assignment problem for system \eqref{2}. To this end, for a given feedback gain matrix $K$, define the corresponding linear operator for  system $(A, B, \bar{A}, \bar{B}) $ as
\begin{align}\label{17}
\mathcal{L}_k: X \mapsto &(A + B K) X (A + B K)' \no\\
&+ (\bar{A} + \bar{B} K) X (\bar{A} + \bar{B} K)', \quad X \in \mathbb{S}^n.
\end{align}
The spectrum of \( \mathcal{L}_k \) is defined as:
\begin{align}\label{18}
\sigma(L_k) = \left\{ \lambda \in  \mathbb{C}:  \mathcal{L}_k(X) = \lambda X, \  X \in \mathbb{S}^n, \ X \neq 0 \right\}.
\end{align}

According to the analysis in Remark \ref{rem1}, we derive the the result of spectrum assignment problem for system \eqref{2} as follows.
\begin{corollary}\label{cor1}
Assume that

i) \( \text{Rank}(\bar{B}) = n \) in system \eqref{2}.

ii)  System \( (H + \alpha L, F) \) is completely controllable for the given $\alpha$ in Definition \ref{def2}, where matrices $ H, L, F $ defined in Remark \ref{rem1}.

Then, the spectrum \eqref{18} generated by system \eqref{2} can be assigned to the desired set $\{ \mu_1^*, \dots, \mu_\frac{n(n+1)}{2}^*\}$ according to Definition \ref{def2}. In this case,
\begin{align}\label{19}
K = \begin{pmatrix} K_u - \bar{A} \\ K_v \end{pmatrix},
\end{align}
where $K_u = \alpha I$, and $K_v$  is given by \eqref{16}.
\end{corollary}

\begin{proof}
According to Theorem \ref{Thm1}, we only need to show that under the condition i), system (\ref{1}) with feedback controller (\ref{6}) can be reformulated as system \eqref{2} with the controller $U(k) =K x(k)$, where $K$ defined by \eqref{19}. In fact, by (\ref{6}) and defining $u(k) = \bar{A} x(k) + q(k)$, we have \begin{align}\label{Q_k}
q(k) = u(k) - \bar{A} x(k) = (K_u - \bar{A} ) x(k),
\end{align}
In view of Remark \ref{rem1}, and by substituting \eqref{Q_k} into $U(k) = \begin{bmatrix} q(k) \\ v(k) \end{bmatrix}$, we obtain that $U(k) =K x(k)$ with the definition \eqref{19}. This completes the proof of Corollary \ref{cor1}.
\end{proof}
\subsection{Learning Algorithm}\label{sec3.2}
In the above part, the feedback gain is given by (\ref{13}) which strongly rely on the system matrix parameters $H, L, F$. However, the exact acquisition of these parameters are difficult \cite{mukherjee2022reinforcement}, which motivates the following learning algorithm to design the feedback gains for spectrum assignment.

To this end, we transform the vector-form system \eqref{7} to the following matrix form:
\begin{align}\label{Xk}
\begin{cases}
X(k+1) \!=\! (H \!+\! L K_u \!+\! F K_v) X(k) \!+\!  W(k) K_u X(k), \\
X(0) = X_0,
\end{cases}
\end{align}
where $X(k)=[x(1), x(2), ..., x(n)],$ $W(k)$ is a diagonal matrix and the main diagonal elements are random variables with a mean of 0 and a variance of $\delta$.

By letting $X_0=I_n,$ we derive the observation of $X_1$ as
 \begin{align}\label{X1}
X(1) = (H  +  L K_u +  F K_v) +  K_u W(0).
\end{align}
Note that $K_u =  \alpha I$ and $G = H + \alpha L$,  the observation of $X(1)$ in \eqref{X1} can be rewritten as
 \begin{align}\label{X2}
X(1) = (G  +  F K_v) +  \alpha  W(0).
\end{align}
Under Assumption \ref{Ass1} and according to \cite{Chen2000pole}, the pole of the system \( (G, F) \)  can be assigned based on the observation of $X(1).$

In the following, we will present the learning algorithm to address the $\alpha$-spectrum assignment problem for stochastic system (\ref{1}) with unknown matrix parameters, which is divided into three steps. The first step is to determine the update of an index function; The second step is to design a stochastic approximation algorithm to learn the feedback gains; The third step is to specify the termination condition of the algorithm.

Firstly, we aim to give the definition of an index function $J(p)$ with initial value $J(0)=0.$ To this end, we define a vector function $K_0(p, j, i)=L_{i-1}$  for $i=1, 2, ..., n+1, j=J(p), J(p)+1, ..., p=0, 1, 2, \ldots,$ where
\begin{align}\label{Alg1}
L_0 = [0 \  0 \cdots 0], \quad L_1 = [1\ 0 \cdots 0], \quad L_2 = [1 \ 1 \cdots 1].
\end{align}
That is, $K_0$ take values in the set $\{L_0, L_1, \dots, L_n\}$ periodically with \( n \)-dimensional row vectors.

We next give the update of the index function $J(p)$ for  $p\geq 1$. By letting  $K_v=K_0(p, j, i)$, we obtain the observation of $X(1)$ from (\ref{X2}), denoted by $X_1(K_0(p,j,i))$, and the corresponding characteristic polynomial denoted by
\begin{align}\label{X4}
&\det(\lambda I - X_1(K_0(p,j,i)))\no\\
 = &\lambda^n + a_1(K_0(p,j,i))\lambda^{n-1} + \cdots +  a_n(K_0(p,j,i)).
\end{align}
Then we derive the coefficients as a row vector:
\begin{align}\label{X5}
& a(K_0(p,j,i)) \no\\
 =& [ a_1(K_0(p,j,i)),  a_2(K_0(p,j,i)), \dots,  a_n(K_0(p,j,i))].
\end{align}
Thus, we have the following matrix $A(p, j+1)$ as
\begin{align}\label{Alg2}
A(p, j+1) = \begin{bmatrix}
a(K_0(p,j,2)) - a(K_0(p,j,1)) \\
\vdots \\
a(K_0(p,j,n+1))   - a(K_0(p,j,n))
\end{bmatrix},
\end{align}
and derive recursively the average of $A(p, j+1)$ as
\begin{align}\label{Alg3}
C(p, j+1) = \frac{j}{j+1} C(p, j) + \frac{1}{j+1}A(p, j+1).
\end{align}
for $j=0, 1, ...$.
Accordingly, we give the update of the index function $J(p)$ as
\begin{align}\label{Alg4}
J(p+1)=\inf\{j: j>J(p), \det C(p, J(p+1)) \neq 0\}.
\end{align}

Secondly, we design the following stochastic approximation algorithm:
\begin{align}\label{Alg33}
K(p+1, s+1) =& K(p+1, s) - \beta (s)[a(K(p+1, s)- a]\no\\
&\times C(J(p+1)),
\end{align}
for $s=1,2, ...$ with arbitrarily given initial value $K(p+1, 0),$ where $a$ denotes the coefficient vector of the characteristic polynomial as shown in \eqref{16}, $a(K(p+1, s))$ represents the following coefficient vector
\begin{align}\label{XX5}
& a(K(p+1, s))\no\\
= &[ a_1(K(p+1, s),  a_2(K(p+1, s)), \dots,  a_n(K(p+1, s))]
\end{align}
satisfying the following characteristic polynomial
\begin{align}\label{XX4}
&\det(\lambda I - X_1(K(p+1, s))\no\\
 = &\lambda^n + a_1(K(p+1, s)\lambda^{n-1} + \cdots +  a_n(K(p+1, s)),
\end{align}
while $X_1(K(p+1, s))$ denotes the observation of $X(1)$ under the feedback gain $K_v=K(p+1, s)$ in (\ref{X2}), $\beta(s)$ denotes a sequence of positive real numbers with
\begin{align}\label{bs}
\sum_{s=1}^{\infty} \beta(s) = \infty, \quad \sum_{s=1}^{\infty} \beta(s)^r < \infty,
\end{align}
 for $r\in(1,2].$

Thirdly, we specify the termination condition of the algorithm (\ref{Alg33}). For the given convergence threshold $\epsilon$, if
\begin{align}\label{condition}
|K(p+1, s+1)-K(p+1, s)| <\epsilon,
\end{align}
then we output the desired feedback gain $K_v=K(p+1, s+1)$. Otherwise, we run the algorithm (\ref{Alg33}) from $s=1$ to $s(p+1)$, where
\begin{align}\label{XX444}
s(p+1) = &\inf\{s : s > 0, \|K(p+1, s) -\beta ( s) [a(K(p+1, s)\no\\
& - a] C(J(p+1))\| > M(p+1)\},
\end{align}
while $M(p)$ denotes a sequence of positive real numbers with
\begin{align}
M(p) \to \infty, \ p \to \infty.\label{Mp1}\\
M(p) < M(p+1), \ \forall p \geq 0.\label{Mp2}
\end{align}
Then let $p=p+1$ and repeat the above procedures until the condition \eqref{condition} is satisfied.

As a conclusion, we present the corresponding algorithm as shown in Algorithm 1 and the main result.

\begin{theorem}\label{Thm2}
Under Assumption \ref{Ass1}, there exists a $p>0$ in Algorithm \ref{algorithmic1} such that $\lim_{s\rightarrow \infty}K(p,s)=K^\star$, where $K(p,s)$ calculated by \eqref{Alg33}, $K^\star$ denotes the desired feedback gain matrix in \eqref{13}.
\end{theorem}
\begin{proof}
Under Assumption \ref{Ass1} and according to Algorithm \ref{algorithmic1}, the conditions of Theorem \ref{Thm2} in \cite{Chen2000pole} are fulfilled, which implies that the convergence of Algorithm \ref{algorithmic1} holds.
\end{proof}

\begin{algorithm}[h!]
\caption{Stochastic approximation-based learning algorithm}
\label{algorithmic1}
\begin{algorithmic}[1]
\STATE \textbf{Input:} Parameters $n$, $\alpha$, $\lambda_i,$ $i=1, ..., n$
\STATE \textbf{Output:} Feedback gain $K_v$ for Problem 1

\STATE Define the $n$-dimensional row vectors $\{L_0, L_1, \dots, L_n\}$ by \eqref{Alg1}, the sequence of positive real numbers $\beta (s)=\frac{1}{s}, M(p)=p,$
the convergence threshold $\epsilon=10^{-8}.$

\STATE Let $J(0)=0$. Start from $p=0$.

\STATE Let $K_0(p, j, i)=L_{i-1}$, $j=J(p), J(p)+1, ...,$ $i=1, 2, ..., n+1.$ Obtain the corresponding observation of $X(1)$ by system with unknown coefficients. Calculate the characteristic polynomial \eqref{X4}, the coefficients $ a(K_0(p,j,i))$ by \eqref{X5}.

\STATE Calculate the matrix $A(j+1)$ by \eqref{Alg2}, the matrix $C(j+1)$ by \eqref{Alg3}. This step start from $j=J(p)$, end for $j=J(p+1)$, which defined by \eqref{Alg4}.

\STATE Let $K(p+1, 0)$ be an arbitrary $n$-dimensional vector. Obtain the corresponding observation of $X(1)$ by system with unknown coefficients. Calculate the characteristic polynomial \eqref{XX4}, the coefficients $ a(K(p+1, s))$ by \eqref{XX5}.

\STATE  Update $K(p+1, s)$ by the stochastic approximation algorithm \eqref{Alg33} for $s=1,2, ...$, until $s(p+1)$ defined by \eqref{XX444}.

\STATE For $p=1, 2, ...$, continue step 5-8 until condition \eqref{condition} is satisfied.
\end{algorithmic}
\end{algorithm}


\section{ Spectrum Assignment for Continuous-time Systems}\label{sec4}
\subsection{Continuous-time Stochastic Systems}
Consider the following continuous-time stochastic systems with multiplicative noise:
\begin{align}\label{c1}
\left\{\aligned
&dx(t) = \left[H x(t) + L u(t) + F v(t)\right]dt + u(t) d\sigma(t)  \\
&x(0) =  x_0,\\
\endaligned\right.
\end{align}
where \( x(t) \in \mathbb{R}^n \) is the state, \( u(t) \in \mathbb{R}^n \) and \( v(t) \in \mathbb{R} \) are control inputs. $ H, L, F$ are constant matrices of compatible dimensions. $\sigma(t)$ is  one-dimensional standard Brownian motion on a complete probability space $(\Omega, \mathcal{F}, P, \mathcal{F}_{t}|_{t\geq0})$.

Similar to Section \ref{sec2.1}, we introduce the following operator:
\begin{align}\label{c8}
\mathcal{L}_c: X \mapsto & (H + L T_u + F T_v) X + X (H + L T_u + F T_v)' \no\\
& +  T_u X T_u', \quad \forall X \in \mathbb{S}^n,
\end{align}
where $T_u, T_v$ are constant matrices with compatible dimensions, and define the spectrum of operator $\mathcal{L}_c$ as the following set:
\begin{align}\label{c9}
\sigma(\mathcal{L}_c) = \left\{ \mu \in \mathbb{C}: \mathcal{L}_c(X) = \mu X, \  X \in \mathbb{S}^n, \  X \neq 0 \right\},
\end{align}
where $\mathbb{C}$ denotes the set of complex numbers. In particular, the spectrum $\sigma(\mathcal{L}_c)$ can be obtained by solving the eigenvalues of the following matrix:
\begin{align}\label{mc9}
&(M' M)^{-1} M' [I_n \otimes (H+L T_u+F T_v)\no\\
&+ (H+L T_u+F T_v) \otimes I_n + T_u \otimes T_u ] M,
\end{align}
where $M$ is defined by $vec(X) = M \overline{vec}(X).$
\begin{remark}
It is noted that all the spectrum of operator $\mathcal{L}_c$ lies in  the open left-hand complex plane is equivalent to the fact that the stochastic
system \eqref{c1}  is mean-square stable under the state feedback controller
 \begin{align}\label{c6}
u(t) &= T_u x(t), \quad v(t) = T_v x(t),
\end{align}
and the corresponding closed-loop system of \eqref{c1} becomes:
\begin{align}\label{c7}
\begin{cases}
dx(t) = (H + L T_u + F T_v) x(t)dt + T_u x(t)d\sigma(t),\\
x(0) = x_0.
\end{cases}
\end{align}
\end{remark}

Then, similar to Section \ref{sec2.2}, we present the definitions of the spectrum assignment and  $\alpha$-spectrum assignment for stochastic system (\ref{c1}).
\begin{definition}\label{defc1}
The spectrum assignment for stochastic system (\ref{c1}) means that for given $\mu_s, s=1,\ldots, \frac{n(n+1)}{2}$, there exist matrices $T_u$ and $T_v$ such that the spectrum of operator $\mathcal{L}_c$ is $\mu_s, s=1, \ldots, \frac{n(n+1)}{2}$.
\end{definition}

\begin{remark}
The spectrum assignment in Definition \ref{defc1} cannot be arbitrarily in general. Taking the scalar system $n=1$ for example, denote the eigenvalue of $H+L T_u+F T_v$ and $T_u$ as $\kappa_1$ and $\kappa_2$, respectively, the spectrum of operator $\mathcal{L}_c$ must be greater than or equal to a certain real number, that is, one cannot find any $T_u$ and $T_v$ such that the spectrum of operator $\mathcal{L}_c$ is $\mu_1<2 \kappa_1 + \kappa_2^2$.
\end{remark}


\begin{definition}\label{defc2}
The $\alpha$-spectrum assignment for stochastic system (\ref{c1}) means that for any $\alpha \in \mathbb{R}$, and $\lambda_i \in \mathbb{R}, i=1,\ldots, n$, there exist matrices $T_v$ and $T_u=\alpha I$ such that for given spectrum $\mu_s^* , s=1, \ldots, \frac{n(n+1)}{2}$ of operator $\mathcal{L}^c$ satisfies $\mu_s^* = \lambda_i +\lambda_j + \alpha^2,\quad  j\geq i, \  i, j = 1, \dots, n$.
\end{definition}

\begin{remark}
According to Definition \ref{defc2}, $0$-spectrum assignment is equivalent to the pole assignment for the deterministic system. In fact, the operator $\mathcal{L}_c$ in the deterministic case is reduced to
\begin{align}\label{dc8}
\mathcal{L}_c: X \mapsto (H+ F T_v) X+X (H+ F T_v)', \quad \forall X \in \mathbb{S}^n .
\end{align}
It is easily obtained that the structure of the spectrum for the operator $\mathcal{L}_c$ must be in the form of $\mu_s=\lambda_i+\lambda_j$, where $\lambda_i, i=1, \ldots, n$ are the eigenvalues of the matrix $H+FK_v$. Accordingly, the spectrum assignment and the pole assignment are equivalent for the deterministic systems.
\end{remark}

Based on the above definitions, the addressed problem is stated as follows.
 \begin{problem}\label{Pc1}
Find $T_v$ to achieve $\alpha$-spectrum assignment for stochastic system (\ref{c1}).
 \end{problem}

In the following, we first present the condition for spectrum assignment and the design of the feedback gain matrices for Problem \ref{Pc1}, and then design a stochastic approximation algorithm for the feedback gain matrices when the system matrices in \eqref{c1} are unknown.
\subsection{Spectrum Assignment}
Similar to Assumption \ref{Ass1}, we make the following controllability assumption for continuous-time systems to guarantee the feasibility of the spectrum assignment.
\begin{assumption}\label{Assc1}
The deterministic system ($ G, F$) is completely controllable, where $G = H + \alpha L$ for the given $\alpha$ in Problem \ref{Pc1} and matrices $L, F$ in system \eqref{c1}.
\end{assumption}

The main result of Problem \ref{Pc1} is given as follows:
 \begin{theorem}\label{Thmc1}
 Under Assumption \ref{Assc1}, the spectrum \eqref{c9} generated by system \eqref{c1} can be assigned to the desired set \( \{ \mu_1^*, \dots, \mu_\frac{n(n+1)}{2}^*\}  \) in definition \ref{defc2}. In this case, the feedback gain \( T_v \) is given by
\begin{align}
T_v =  a D + b G^n,\label{c13}
\end{align}
where $a=\begin{bmatrix}a_{n-1}^* & \cdots & a_1^* & a_0^*\end{bmatrix}  $ denotes the coefficient of the following polynomial:
\begin{align}
&\prod_{i=1}^{n} (\beta  - \lambda_i) = \beta^n + a_{n-1}^* \beta^{n-1} + \cdots + a_1^* \beta + a_0^*, \label{c16}
\end{align}
and
\begin{align}
b =&- \begin{bmatrix} 0&  \cdots & 0& 1\end{bmatrix}\begin{bmatrix} F& G F&  \dots &G^{n-1} F\end{bmatrix}^{-1}, \label{c14}\\
D=& \begin{bmatrix} (b G^{n-1})' &  \cdots & (b G)'& b'\end{bmatrix}'.\label{c166}
\end{align}

 \end{theorem}

\begin{proof}
The proof is similar to Theorem \ref{Thm1}, and thus be omitted here.
\end{proof}

According to Theorem \ref{Thmc1}, we further derive the result of spectrum assignment problem for the standard continuous-time linear stochastic system: 
\begin{align}\label{c2}
\left\{\aligned
&dx(t) = \left[Ax(t) + Bu(t)\right]dt + \left[\bar{A}x(t) + \bar{B}u(t)\right]d\sigma(t), \\
&x(0)= x_0,
\endaligned\right.
\end{align}
when the rank of the matrix $\bar{B}$ is equal to $n$.
To this end, for a given feedback gain matrix $T$, define the corresponding linear operator for  system \eqref{c2} as
\begin{align}\label{c17}
\mathcal{L}_{c,T}: X \mapsto &(A + B T) X+X (A + B T)' \no\\
&+ (\bar{A} + \bar{B} T) X (\bar{A} + \bar{B} T)', \quad X \in \mathbb{S}^n,
\end{align}
and the spectrum of \( \mathcal{L}_{c,T} \) as:
\begin{align}\label{c18}
\sigma(\mathcal{L}_{c,T}) = \left\{ \lambda \in  \mathbb{C}:  \mathcal{L}_{c,T}(X) = \lambda X, \  X \in \mathbb{S}^n, \ X \neq 0 \right\}.
\end{align}
Then we have the result of spectrum assignment problem for system \eqref{c2} as follows.
\begin{corollary}\label{corc1}
Assume that

i) \( \text{Rank}(\bar{B}) = n \) in system \eqref{c2}.

ii) System \( (H + \alpha L, F) \) is completely controllable for the given $\alpha$ in Definition \ref{defc2}, where matrices $ H, L, F $ defined in Remark \ref{rem1}.

Then, the spectrum \eqref{c18} generated by system \eqref{c2} can be assigned to the desired set $\{ \mu_1^*, \dots, \mu_\frac{n(n+1)}{2}^*\}$ in Definition \ref{defc2}.

In this case,
\begin{align}\label{c19}
T= \begin{pmatrix} T_u - \bar{A} \\ T_v \end{pmatrix},
\end{align}
where $T_u =  \alpha I$, $T_v$ is given by \eqref{c13}.
\end{corollary}

\begin{proof}
 The proof is similar to Corollary \ref{corc1}, and thus be omitted here.
\end{proof}
\subsection{Learning Algorithm}
When $H, L, F$ are unknown in system \eqref{c1}, we present a learning algorithm in this part to design the feedback gains for spectrum assignment.

To  this end, for a sufficiently small time interval $[0, \delta t]$, we denote the state of system \eqref{c7} at time $\delta t$ as $x(\delta t)$.
By letting the initial values as
\begin{align}
x_1(0)=&\frac{1}{\delta t}[1 \  0\  0 \cdots 0], \label{IV1}\\
x_2(0)=&\frac{1}{\delta t}[0 \  1 \ 0 \cdots 0], \\
&...\no\\
x_n(0)=&\frac{1}{\delta t}[0 \cdots  0\ 1],\label{IVn}
\end{align}
we can obtain the current states on time $\delta t$ from system  \eqref{c7} with the feedback gain $T_u =  \alpha I$ and the given feedback gain $T_v$, and denote as $x_1(\delta t, T_v),$ $x_2(\delta t, T_v), ...,$ $x_n(\delta t, T_v)$. Further, denote the observation $x(\delta t, T_v)=[x_1(\delta t, T_v), x_2(\delta t, T_v), ..., x_n(\delta t, T_v)],$
and calculate a new observation $Y_1(T_v)=x(\delta t, T_v)-\frac{1}{\delta t} I_n $,
then the characteristic polynomial for $Y_1(T_v)$ and the corresponding coefficients are defined by
\begin{align}\label{XC4}
&\det(\lambda I - Y_1(T_v))\no\\
 = &\lambda^n + a_1(T_v)\lambda^{n-1} + \cdots +  a_n(T_v),
\end{align}
and
\begin{align}\label{XC5}
 a_c(T_v) = [ a_1(T_v),  a_2(T_v), \dots,  a_n(T_v)],
\end{align}
respectively. 

Similar to Section \ref{sec3.2}, we next present the learning algorithm, which is divided into the following three steps.

Firstly, definite an index function $J_c(p)$ with initial value $J_c(0)=0.$
By defining $T_0(p, j, i)=L_{i-1}$  for $i=1, 2, ..., n+1,$ and letting $T_v=T_0(p, j, i)$ in system \eqref{c1}, we can obtain the observation $Y_1(T_0(p, j, i))$ on time $\delta t$, the corresponding characteristic polynomial $\det(\lambda I - Y_1(T_0(p, j, i)))$, and the coefficients $ a(T_0(p, j, i))$ according to \eqref{IV1}-\eqref{XC5}.
Then we have the matrix $A_c(p, j+1)$ as
\begin{align}\label{Alggg2}
A_c(p, j+1) = \begin{bmatrix}
a((T_0(p, j, 2))) - a((T_0(p, j, 1))) \\
\vdots \\
a((T_0(p, j, n+1 )))   - a((T_0(p, j, n)))
\end{bmatrix},
\end{align}
and recursively compute the average of $A_c(p, j+1)$ as
\begin{align}\label{Alggg3}
C_c(p, j+1) = \frac{j}{j+1} C_c(p, j) + \frac{1}{j+1}A_c(p, j+1).
\end{align}
for $j=0, 1, ...$.  Therefore, we define the update of the index function $J_c(p)$ as 
\begin{align}\label{Alggg4}
J_c(p+1)=\inf\{j: j>J_c(p), \det C_c(J_c(p+1)) \neq 0\}.
\end{align}


Secondly, we design the following stochastic approximation algorithm: 
\begin{align}\label{Alggg33}
T(p+1, s+1) =& T(p+1, s) - \beta (s)[a_c(T(p+1, s)- a_c]\no\\
&\times C_c(J_c(p+1)), 
\end{align}
for $s=1,2, ...$ with arbitrary given initial value $T(p+1, 0),$ where $ \beta (s)$ satisfies \eqref{bs}, $a_c(T(p+1, s))$ and $a_c$ denote the coefficient vectors of \eqref{XC5} with $T_v=T(p+1, s)$ and \eqref{c16}, respectively.

Thirdly, we specify the termination condition of the algorithm \eqref{Alggg33}. For the given convergence threshold $\epsilon$, if
\begin{align}\label{ccondition}
|T(p+1, s+1)-T(p+1, s)| <\epsilon,
\end{align}
then we output the desired feedback gain $T_v=T(p+1, s+1)$. Otherwise, we run the algorithm \eqref{Alggg33} from $s = 1$ to $s(p+1)$, where
\begin{align}\label{XXC444}
s(p+1) = &\inf\{s : s > 0, \| T(p+1, s) - \beta (s)[a_c(T(p+1, s)\no\\
& - a_c] C_c(J_c(p+1))\| > M(p+1)\}.
\end{align}
with $M(p)$ satisfying \eqref{Mp1}-\eqref{Mp2}. Then letting $p=p+1$ and repeat the above procedures until condition \eqref{ccondition} is satisfied.


\begin{algorithm}[h!]
\caption{Stochastic approximation-based learning algorithm}
\label{algorithmic2}
\begin{algorithmic}[1]
\STATE \textbf{Input:} Parameters $\delta t$, $n$, $\alpha$, $\lambda_i,$ $i=1, ..., n$
\STATE \textbf{Output:} Feedback gain $T_v$ for Problem 2

\STATE Define the $n$ initial values $x_1(0), x_2(0), \dots, x_n(0)$ by \eqref{IV1}-\eqref{IVn}, the sequence of positive real numbers $\beta (s)=\frac{1}{s}, M(p)=p,$ the convergence threshold $\epsilon=10^{-8}.$

\STATE Let $J_c(0)=0$. Start from $p=0$.

\STATE Let $T_0(p, j, i)=L_{i-1}$, $j=J_c(p), J_c(p)+1, ...,$ $i=1, 2, ..., n+1.$ Obtain the corresponding observation $x(\delta t, T_s)$ by the continuous-time system with unknown coefficients. Define $Y_1(T_s)=x(\delta t, T_s)-\frac{1}{\delta t} I_n $, and calculate the corresponding characteristic polynomial \eqref{XC4}, the coefficients $ a_c(T_0(p,j,i))$ by \eqref{XC5}.

\STATE Calculate the matrix $A_c(j+1)$ by \eqref{Alggg2}, the matrix $C_c(j+1)$ by \eqref{Alggg3}. This step start from $j=J_c(p)$, end for $j=J_c(p+1)$, which defined by \eqref{Alggg4}.

\STATE Let $T(p+1, 0)$ be an arbitrary $n$-dimensional vector. Obtain the corresponding observation $x(\delta t, T_s)$ by the continuous-time system with unknown coefficients. Calculate $Y_1(T_s)=x(\delta t, T_s)-\frac{1}{\delta t} I_n $, the characteristic polynomial \eqref{XC4}, the coefficients $ a(T(p+1, s))$ by \eqref{XC5}.

\STATE  Update $T(p+1, s)$ by the stochastic approximation algorithm \eqref{Alggg33} for $s=1,2, ...$, until $s(p+1)$ defined by \eqref{XXC444}.

\STATE For $p=1, 2, ...$, continue step 5-8 until condition \eqref{ccondition} is satisfied. 
\end{algorithmic}
\end{algorithm}

Accordingly, we present the corresponding algorithm as shown in Algorithm  \ref{algorithmic2} and the effectiveness result in the following Theorem.

\begin{theorem}\label{Thm4}
Under Assumption \ref{Ass1}, there exists a $p>0$ in Algorithm \ref{algorithmic2} such that $\lim_{s\rightarrow \infty}T(p,s)=T^\star$, where $T(p,s)$ calculated by \eqref{Alggg33}, $T^\star$ denotes the desired feedback gain matrix in \eqref{c13}.
\end{theorem}
\begin{proof}
Notice that the state $x(\delta t)$ of system \eqref{c7} can be estimated as 
\begin{align}\label{Xc1}
\begin{cases}
x(\delta t) =[I_n + (G +  F T_v)\delta t +  \alpha w(\delta t)]x(0),\\
x(0) = x_0,
\end{cases}
\end{align}
under the feedback gain $T_u =  \alpha I$, where 
 $w(\delta t)$ represents the random variable with a mean of $0$ and a variance of $\delta t.$ For the initial values \eqref{IV1}-\eqref{IVn} and the given feedback gain $T_s$, by defining $Y_1(T_s)=x(\delta t, T_s)-\frac{1}{\delta t} I_n $,
we obtain that $Y_1(T_s)$ can be estimated as
 \begin{align}\label{Xc2}
Y_1(T_s)= (G  +  F T_s) +\frac{\alpha}{\delta t}  W_s(\delta t),
\end{align}
where diagonal matrix $W_s(\delta t), s=0, 1, ...$ are mutually independent, and the diagonal entries of each matrix $W_s(\delta t)$ are also mutually independent random variables with a mean of 0 and a variance of $\delta t$. This implies that the mean of $\frac{\alpha}{\delta t}  W_s(\delta t)$ is 0 and the variance is $\frac{\alpha^2}{\delta t}$. Thus, the role of $Y_1(T_s)$ in Algorithm \ref{algorithmic2} is equivalent to the observation of $X(1)$ in Algorithm \ref{algorithmic1} that satisfies  \eqref{X2}. Therefore, the result follows similar to Theorem \ref{Thm2}.
\end{proof}

\section{Numerical Examples}\label{sec5}
In this section, we provide numerical examples to show the effectiveness of the designed algorithms.
\subsection{Example 1}
Consider the discrete-time stochastic system \eqref{1} with parameters
$n=3,$
$H=  \begin{bmatrix}
-5 & 5 & -2\\
-4 & 3 & -1\\
6 & -4 & 5
\end{bmatrix},$
$L=\begin{bmatrix}
-5 & 5 & -2 \\
-4 & 3 & -1 \\
6 & -4 & -5
\end{bmatrix},$
$F=\begin{bmatrix}
1 \\ 0.5 \\ -1
\end{bmatrix},$ $\delta=0.01.$
Given the desired spectrum $\mu^* $ in Definition \ref{def2} with $\lambda_1=1 + i,  \lambda_2=1 - i, \lambda_3=3, \alpha=0.1$. According to Theorem \ref{Thm1} and based on the above parameters, we can obtain the feedback gain matrix $K_v$, which ensures the spectrum to the desired location $\mu^*$, as
\begin{align*}
K_v=\begin{bmatrix}
6 & -4 & 2
\end{bmatrix}.
\end{align*}

Next, consider the case with unknown parameters $H, L, F$. By using Algorithm \ref{algorithmic1}, we can obtain the feedback gain matrix $K_v$ and the termination index $p$.
Then we run the algorithm three times and obtain that 
\begin{align*}
&K_v=\begin{bmatrix}
6.0028 & -3.9998 & 2.0063
\end{bmatrix},\quad p=1239,\\
&K_v=\begin{bmatrix}
5.9992 & -4.0000 & 1.9985
\end{bmatrix},\quad p=356,\\
&K_v=\begin{bmatrix}
6.0004 & -3.9997 & 2.0005
\end{bmatrix},\quad p=497,
\end{align*}
It is found that all the errors between $K_v$ and $K^*$ are less than $10^{-2}$. This indicates the effectiveness of Algorithm \ref{algorithmic1}.

\subsection{Example 2}
Consider the continuous-time stochastic system \eqref{1} with parameters
$n=1,$
$H= 21.6,$
$L=24,$
$F=1,$
Given the desired spectrum $\mu^*$ in Definition \ref{defc2} with $\lambda_1=30, \alpha=0.1$. In this case,  it is easy to verify that the target feedback gain matrix $T^*$
According to Theorem \ref{Thmc1} and based on the above parameters, we can obtain the feedback gain matrix $T_v=6,$ which ensures the spectrum to the desired location $\mu^*$.

When the parameters $H, L, F$ are unknown, by running Algorithm \ref{algorithmic2} with $\delta t=0.1$ for three times, we can obtain the feedback gain matrix $T_v$ and the value of $p$ as follows.
\begin{align*}
&T_v=6.0094, \quad p=7,\\
&T_v=6.0001,\quad p=7,\\
&T_v=6.0029,\quad p=8,
\end{align*}
Obviously, all the errors between  $T_v$ and $T^*$ are less than $10^{-2}$. This indicates the effectiveness of Algorithm \ref{algorithmic2}.

\section{Conclusion}\label{sec6}
In this paper,  the spectrum assignment for a class of stochastic systems with multiplicative noise has been studied. By proposing a novel $\alpha$-spectrum assignment for discrete-time and continuous-time stochastic systems, the conditions for $\alpha$-spectrum assignment and the feedback controllers based on the system parameters have been presented, and the stochastic approximation algorithms have been designed to learn the feedback gains
ensuring the spectrum of the stochastic systems to achieve the predetermined value. Numerical examples have demonstrated the effectiveness of the proposed algorithm.

\bibliographystyle{IEEEtran}
\bibliography{1}

\end{document}